\documentclass[a4paper,11pt,reqno]{amsart}

\usepackage[T2A]{fontenc}
\usepackage[cp1251]{inputenc}
\usepackage[english]{babel}
\usepackage {amsmath, amsthm, amsfonts, amssymb}
\usepackage{enumerate}
\theoremstyle{plain}
\newtheorem{theorem}{Theorem}

\newtheorem{corollary}{Corollary}
\newtheorem*{corollary*}{��������}

\theoremstyle{definition}
\newtheorem{definition}{Definition}

\theoremstyle{remark}

\newtheorem*{remark*}{Remark}


\begin{document}

\title[On faithfulness, DP-transformations and Cantor series expansions]
{On faithfulness and DP-transformations generated by arithmetic Cantor series expansions}

\author[G. Torbin, Yu. Voloshyn  ]{ Grygoriy Torbin $^{1,2}$, Yuliia Voloshyn $^{3}$} 
\date{}

\maketitle
\begin{abstract}
The paper is devoted to the study of conditions for the Hausdorff-Besicovitch faithfulness of the family of cylinders generated by Cantor series expansions. 
We show that there exist subgeometric Cantor series expansions for which the corresponding families of cylinders are not faithful for the Hausdorff-Besicovitch dimension on the unit interval.  On the other hand we found a rather
 wide subfamily of subgeometric Cantor series expansions  generating faithful families of cylinders.  

We also study conditions  for the Hausdorff-Besicovitch dimension   preservation on [0;1] by probability distribution functions of random variables with independent symbols of arithmetic Cantor series expansions. 

\end{abstract}
$^{1}$~ Dragomanov Ukrainian State University, Pyrogova str. 9, 01601 Kyiv
(Ukraine) 
$^{2}$ Institute for Mathematics of NASU, Tereshchenkivs'ka str. 3, 01601 Kyiv (Ukraine); E-mail:
g.m.torbin@udu.edu.ua (corresponding author)

$^{3}$~ Dragomanov Ukrainian State University, Pyrogova str. 9, 01601 Kyiv
(Ukraine)
\medskip

\medskip


\section{Introduction}

The notion of Hausdorff-Besicovitch dimension is widely known and plays an important role in both mathematics and applied research (see, e.g., \cite{BGZ, Falconer 1990, Mandelbrot 1983, Pesin 1997, Triebel 1997}). However, its calculation or even estimation is a rather non-trivial problem (\cite{ABPT, AKNT, Baranski 2007, Billingsley 1961, Falconer 1990, NT_TVIMS12, Pratsiovytyi 1998}).

Various approaches and special techniques for computing the Hausdorff-Besicovitch dimension are described in detail in \cite{Falconer 1990, Falconer 1997, Mandelbrot 1983}. In particular, an approach based on the theory of DP-transformations was presented in \cite{APT1, APT2} and developed in \cite{AGIT, GNT, IT 2}.

\begin{definition}
A bijective function $f(x)$: $[0;1]\rightarrow[0;1]$ is called a DP-transformation on $[0;1]$ if 
$$
 \forall E \subset  [0;1]: \quad \dim_{H}(E)=\dim_{H}(f(E)).
$$
\end{definition}

Later in the work \cite{AILT}, an alternative approach was presented that is closely related to the theory of DP-transformations and is also based on the notion of faithful families of coverings, which significantly simplifies the calculation of the Hausdorff-Besicovitch dimension for a given set.

The study of DP-transformations is important for two main reasons \cite{APT1}:
if a DP-transformation maps a set $E$ to a set $E'$ and preserves the Hausdorff-Besicovitch dimension, then it suffies to calculate the dimension of a simpler set.
Fractal geometry can be considered as the study of invariants of the group of DP-transformations of a space. So, fractal geometry can be considered as  a generalization of affine geometry, as the latter discipline investigates the invariants of the affine transformation group, which forms a subgroup of the group of DP-transformations.

This work is devoted to the study of properties of random variables with independent symbols of Cantor arithmetic expansions and their applications in the theory of DP-transformations.
 
In Section 2 we show that for the subgeometric case (i.e., if the basic sequence $\{n_k\}$ grows at most geometrically), the corresponding family of cylinders can be non-faithful. We also prove sufficent conditions for the  family of cylinders generated by subgeometric  Cantor expansions to be faithful. In particular, if the basic sequence $\{n_k\}$ satisfies the following conditions 
$$
\text{where } \quad a_k \leq n_k \leq b_k, \quad \forall k \in N,
$$
with  $\{a_k\}$ being an arithmetic progression such that  $a_1 \geq 2$ and $d \geq 1,$

and $\{b_k\}$ being a geometric progression such that $b_1 \geq 2$ and $q \geq 1,$
then the corresponding family of cylinders as faithful. 

Section 3 is devoted to the study of properties of random variables with independent symbols of Cantor expansions. Specifically, it provides necessary and sufficient conditions for the probability  distribution functions of  random variables with independent symbols of Cantor expansions to be in to the DP-class, under the condition that $\{n_k\}$ is bounded and probabilities $p_{ik}$ are separated from zero.

Section 4 shows that the well-known necessary conditions for the distribution function of a random variable with independent Cantor symbols to belong to DP-transformations are not sufficient, where the sequence $\{n_k\}$ forms an arithmetic progression.

\section{ On some faithful families of coverings for the Hausdorff-Besicovich dimension calculation generated by Cantor series expansions}

Let $E \subset  [0;1]$ and let $\Phi$  be some family of subsets from this segment.

\begin{definition}
A family $\Phi$ of subsets of $[0;1]$ is said to be locally fine if for any $E\subset [0;1] $ there exists an at most countable $\varepsilon$-covering $\{E_j\}$ of $E$ , $E_j\in \Phi$.
\end{definition}

Recall that the Hausdorff-Besicovitch dimension of a set $E\subset [0;1]$ with respect to $\Phi$ is the number
$$
\dim_{H}(E,\Phi)=\inf\{\alpha:H^{\alpha}(E,\Phi)=0\},
$$

where $H^{\alpha}(E,\Phi):=\lim\limits_{\varepsilon\to0}H^{\alpha}_{\varepsilon} (E,\Phi)=\lim\limits_{\varepsilon\to 0}(\inf\limits_{|E_{k}|\le\varepsilon}(\sum_{k}|E_{k}|^\alpha)$,
\\
where infimum is taken over all possible $\varepsilon$-covering of E by subsets $E_k$ from $\Phi$. 

\begin{definition}
A locally fine covering family $\Phi$ is said to be faithful for the calculation of the Hausdorff-Besicovitch dimension on $[0;1]$ , if
$$
\dim_{H}(E)=\dim_{H}(E,\Phi) , \quad  \forall E \subset  [0;1].
$$
\end{definition}

The problem on necessary and sufficient conditions for the faithfulness certain locally fine systems of coverings has been the subject of research by many scientists \cite{AT2, Besicovitch 1952, Culter 1988, TP}. In particular, an important contribution was made by A. S. Besicovitch, who first proved the faithfulness  of systems of cylinders of binary expansion \cite{Besicovitch 1952} . Later, the faithfulness  of various covering systems was studied by:
Patrick Billingsley (for families of $s$-adic cylinders \cite{Billingsley 1961});
Mykola Pratsiovytyi (for families of $Q$-cylinders \cite{TP});
S. Albeverio, M. Pratsiovytyi, G. Torbin, M. Ibrahim, V. Vasylenko (for families of cylinders of $Q^*$-expansion \cite{IT 1, VVPT});
S. Albeverio, Y. Kondratiev, R. Nikiforov, O. Smiyan, G. Torbin (for families of cylinders of $Q_\infty$-expansion \cite{AKNT});
G. Torbin, V. Vasylenko (for families of cylinders of $\tilde{Q}$-expansion \cite{VMT}).
Necessary and sufficient conditions for the family of Cantor series cylinders to be faithful were found in \cite{AILT}.
\begin{theorem}
The family $\Phi(C)$ of cylinders of the Cantor expansion is  faithful of the Hausdorff-Besicovitch dimension calculation on $[0; 1]$ if and only if the following condition holds:
$$
\lim_{k\to\infty} \frac{\ln(n_k)}{\ln(n_1\cdot n_2 \cdot ... \cdot n_{k-1} )}=0
$$
\end{theorem}
By using this result it if easy to produce examples of faithful as well as non-faithful families $\Phi(C)$ of Cantor expansion cylinders:

a) if $n_k =2^{2^k}$, then the family $\Phi(C)$ is non-faithful;

b) if $\{n_k\}$ is bounded, then the family $\Phi(C)$ is faithful.

During several years the following conjecture was dominated:  if the sequence $\{n_k\}$ is subgeometric (i.e., there exists a positive integer $q$ such that $n_k \le q^k, \forall k\in N$), then the family $\Phi(C)$ is faithful.

Unfortunately this conjecture fails to be true. The simplest counterexample can be produced as follows:
let
$$
n_k = \begin{cases}
2, \ \text{for} \ k\neq 10^s ;
\\
10^k, \ \text{for} \ k=10^s, \  s \in N.
\end{cases}
$$

In such a case $\{n_k\}$ is subgeometric ( $n_k \le 10^k$), but the limit $\lim\limits_{k\to\infty} \frac{\ln(n_k)}{\ln(n_1\cdot n_2 \cdot ... \cdot n_{k-1} )}$ does not equal 0.

The following theorem gives sufficient conditions for subgeometric families of Cantor series cylinders to be faithful.

\begin{theorem}
 Let the basic sequence $\{n_k\}$ satisfies the following cindition:
$$
a_k \le n_k \le b_k ,\quad  \forall k \in N ,
$$
where $\{a_k\}$ forms an arithmetic progression with $a_1\ge2 , d\ge1$,

and $\{b_k\}$ forms a geometric progression with $b_1\ge2 , q\ge1$, 

then the family $\Phi(C)$ of Cantor series cylinders is faithful for the Hausdorff-Besicovitch dimension calculation on $[0;1]$.
\end{theorem}

\begin{proof}
Consider the expression
$$
\frac{\ln(n_k)}{\ln(n_1\cdot n_2 \cdot ... \cdot n_{k-1} )} \le \frac{\ln(b_k)}{\ln(a_1\cdot a_2 \cdot ... \cdot a_{k-1} )}=
$$
$$
=\frac{\ln(b_1 q^{k-1})}{\ln(a_1\cdot (a_1 + d)\cdot ... \cdot (a_1 + (k-2)d)} \le \frac{\ln(b_1) + (k-1)\ln(q)}{\ln(2\cdot(2+1)\cdot...\cdot(k-2))} =
$$
$$
 =\frac{\ln(b_1) + (k-1)\ln(q)}{\ln(k-2)!} \le  \frac{\ln(b_1) + (k-1)\ln(q)}{\ln(\sqrt{2\pi(k-2)}\cdot (\frac{k-2}{e})^{k-2}\cdot e^{\theta_{k-2}})},
$$
where $|\theta_{k-2}|\le \frac{1}{12(k-2)}$.

So,
$$
\frac{\ln(n_k)}{\ln(n_1\cdot n_2 \cdot ... \cdot n_{k-1} )} \le \frac{\ln(b_1) + (k-1)\ln(q)}{\ln(\sqrt{2\pi(k-2)}+ (k-2)\ln(\frac{k-2}{e}) + \theta_{k-2})} \to 0 \ ( \text{as} \ k \to \infty)
$$

Therefore,
$$
 \lim_{k\to\infty} \frac{\ln(n_k)}{\ln(n_1\cdot n_2 \cdot ... \cdot n_{k-1} )}=0
$$ 
Taking into account results from \cite{AILT}, we get the faithfulness of $\Phi(C)$.
\end{proof}

\begin{corollary}
If $\{n_k \}$ is strictly increasing and subgeometric (i.e., there exists a constant $q$ such that $n_k \le q^k , \forall k \in N $), then $\Phi(C)$ is faithful for the Hausdorff-Besicovitch dimension calculation on $[0;1]$.
\end{corollary}

\begin{proof}
Since $n_1\ge2$ and $\{n_k \}$ is increasing, we have $n_k\ge k+1$. Therefore,
$$
a_k \le n_k \le b_k  , \quad  \forall k \in N  ,
$$
where $a_k =k+1$,  $b_k=q^k$.
 
The faithfulness of $\Phi$ follows from the previous theorem.
\end{proof}

\section{DP-transformations generated by Cantor series expansions}

Despite the fact that Cantor series expansions are natural generalizations of $s$-adic representations, the vast majority of problems that are completely solved for $s$-adic representations are still very far from being solved for Cantor series expansions. In particular, the problem of finding necessary and sufficient conditions for the distribution function $F_\xi$ to belong to the DP-class, i.e. to preserve the Hausdorff-Besicovitch dimension of an arbitrary subset on $[0;1]$. Important steps in the study of this problem have been made in the works of M. V. Pratsiovytyi, G. M. Torbin and their students.

Before presenting the new results of our study, let us recall the following definitions.

\begin{definition}
Let $\xi_k$ be a sequence of independent random variables taking values $0, 1, ... ,n_k -1$ with probabilities $ p_{0k}, p_{1k}, ... ,p_{(n_k-1)k}$ correspondently. The random variable 
$$
\xi=\sum_{k=1}^{\infty}   \frac{ \xi_{k}}{n_{1}\cdot n_{2}\cdot ...\cdot n_{k}}   =:  \Delta_{\xi_1\xi_2...\xi_k...}^{C}
$$
is said to be the random variable with independent symbols of Cantor expansion.
\end{definition}

\begin{definition}
A number 
$$
\dim_{H} {\mu_\psi} =\inf_{\mu_\psi(E) =1}  \{ \dim_H E\}
$$
is said to be the Hausdorff dimension of the measure $\mu_\psi$.
\end{definition}

\begin{definition}
 The spectrum of a random variable $\psi$ is the set
$$
S_\psi :=\{ x: F_\psi (x+\varepsilon ) - F_\psi (x-\varepsilon) >0, \forall \varepsilon >0 \}
$$
i.e. , $S_\psi$ is the minimal closed support of the measure $\mu_\psi$.
\end{definition}
Properties of Cantor expansions and properties of random variable $\xi$ were studied by M. Pratsiovytyi, G. Torbin, M. Lebid, B. Mance, R. Nikiforov and other authors.

A fundamentally important breakthrough in the development of the metric and dimensional theory of Cantor expansion was made in \cite{AILT}, where, in addition to the criterion for the faithfulness of the system of cylinders of the Cantor series expansions for the calculating of the Hausdorff-Besicovitch dimension on $[0;1]$, formulae for the calculating the Hausdorff-Besicovitch dimension of the spectrum of measure $\mu_\xi$  was derived. In the same paper authors proved the exact formulae for the calculation the Hausdorff dimension of the measure $\mu_\xi$ itself under the following restriction: 
$$
\sum_{k=1}^{\infty} \left( \frac{ \ln{n_k}} {\ln (n_{1}\cdot n_{2}\cdot...\cdot n_{k-1})}  \right) ^2 <+\infty,
$$
where $\{n_k\}$ is the basic sequence that determines the Cantor expansion.

The study of DP-properties of distribution functions $F_\xi$ at this moment is limited only to cases when the sequence $\{n_k\}$ is bounded  \cite{AILT}. In particular, the following fact has been proven.

\begin{theorem}
If $\{n_k\}$ is bounded and there exists a constant $p_0>0: p_{ik} \ge p_0$, then the probability distribution function of random variable $\xi$  with independent symbols of Cantor series expansion  is DP-transformation if and only if
$$
\dim_{H}{\mu_\xi}=1.
$$ 
\end{theorem}

\section{Counterexample related to DP-transformations generated by arithmetic Cantor series expansions}

If the sequence $\{n_k\}$ is unbounded (for example, when $\{n_k\}$ forms an arithmetic progression), the condition of separation of $p_{ik}$ from zero is impossible to fulfill since $ \min\limits_{i} p_{ik} \le \frac{1}{n_k} $ and if $\{n_k\}$ is unbounded, the sequence $\{\frac{1}{n_k}\}$ has a subsequence tending to 0. Therefore, the previous theorem cannot be applied to the class of unbounded sequences $\{n_k\}$.

At the same time, we note that the general necessary conditions for $F_\xi$ to belong to the DP-class are the following ones:

$1) p_{ik} >0, \quad \forall i \in \{0, ..., n_k -1\};$

$2) \dim_{H}{\mu_\xi}=1.$

Let us construct a counterexample that demonstrates that for arithmetic Cantor expansions, even the simultaneous fulfillment of the above two conditions is not sufficient for $F_\xi$ to belong to the DP-class.

\textbf{Example 1.}Let $n_k =k+1$ and random variable $\xi$:
$$
\xi = \sum_{k=1}^{\infty} \frac {\xi_k}{(k+1)!}  , 
$$
where  \quad 

\begin{tabular}{|c|c|c|c|c|}
\hline
{$\xi_k$} & 0 & 1 & ... & $k$ \\
\hline
 & $\frac {1}{k+1}$ & $\frac {1}{k+1}$ & ...& $\frac {1}{k+1}$  \\
\hline
\end{tabular}
\quad $\forall k\in A:=\{n:n\ne 10^s , s\in N\}$
\\
\\

\begin{tabular}{|c|c|c|c|c|}
\hline
{$\xi_k$} & 0 & 1 & ... & $k$ \\
\hline
 & $ \frac {1}{10^{10^{10^{k}}}} $ & $ \frac {1-\frac {1}{10^{10^{10^{k}}}}}{k} $ & ...& $ \frac {1-\frac {1}{10^{10^{10^{k}}}}}{k}$ \\
\hline
\end{tabular}
 \quad $\forall k\in \bar{A}=B$
\\
\\
Lets us check whether  $\dim_{H}{\mu_\xi}=1$ .

Since
$$
\sum_{k=1}^{\infty} \left( \frac{\ln{n_k}}{\ln (n_{1}\cdot n_{2}\cdot...\cdot n_{k-1})} \right)^2 =\sum_{k=1}^{\infty} \left( \frac {\ln({k+1})} {\ln({k!})} \right)^2 <+\infty ,
$$
we can calculate the Hausdorff dimension of the measure by the following formula \cite{AILT}:
$$
\dim_{H}\mu_{\xi}=\varliminf_{k\to\infty} \frac {h_{1}+h_{2}+...+h_{k}} {\ln ( n_{1}\cdot n_{2}\cdot...\cdot n_{k} ) }, 
$$
where $h_k$ is the entropy of the random variable $\xi_k$, i.e. $h_k=-\sum\limits_{i=0}^{n_k-1} p_{ik}\ln{p_{ik}}$

If $ k \in A$, then $h_k =\ln{k+1}$.

If $ k\in \bar{A}=B$, then $h_k= -\left(\frac{1}{10^{10^{10^{k}}}} \ln{\frac{1}{10^{10^{10^{k}}}}}+k\cdot \frac{1-\frac{1}{10^{10^{10^{k}}}}} {k} \ln{\frac{1-\frac{1}{10^{10^{10^{k}}}}} {k}} \right) \sim\ln k$

because $\lim\limits_{x\rightarrow 0+} x\ln{x}=0$.

Therefore, 
$$
\dim_{H}\mu_{\xi}=\varliminf_{k\to\infty} \frac {h_{1}+h_{2}+...+h_{k}} {\ln ( n_{1}\cdot n_{2}\cdot...\cdot n_{k} ) }=
$$
$$
=\varliminf_{k\to\infty} \frac {\ln{(2\cdot3\cdot4\cdot...\cdot9\cdot10\cdot10\cdot12\cdot...\cdot100\cdot100\cdot102\cdot...\cdot 10^{k}\cdot10^{k})}} {\ln{(2\cdot3\cdot4\cdot...\cdot9\cdot10\cdot11\cdot12\cdot...\cdot100\cdot101\cdot102\cdot...\cdot 10^{k}\cdot(10^{k}+1))}}=
$$
$$
=\varliminf_{k\to\infty} \frac {\ln{(2\cdot3\cdot4\cdot...\cdot10^{k}\cdot(10^{k}+1))}-\ln{(\frac{11
}{10}\cdot\frac{101}{100}\cdot...\cdot\frac{10^{k}+1}{10^{k}})}} {\ln{(2\cdot3\cdot4\cdot...\cdot 10^{k}\cdot(10^{k}+1))}}=
$$
$$
=\varliminf_{k\to\infty} 1- \frac{\ln{((1+\frac{1}{10})\cdot(1+\frac{1}{10^2})\cdot...\cdot(1+\frac{1}{10^k}))}}{\ln{(2\cdot3\cdot4\cdot...\cdot 10^{k}\cdot(10^{k}+1))}}= 1 ,
$$
because $ \prod\limits_{k=1}^{\infty} (1+\frac{1}{10^k})$ is convergent.

Hence, $\dim_{H} \mu_\xi =1$ .

Consider the set $V$:
$$
V= \{x:x=\Delta_{\alpha_1 ... \alpha_9 0\alpha_{11}... \alpha_{99}  0 \alpha_{101} ...}^{C}\}=
$$
$$
= \{x:x=\Delta_{\alpha_1 \alpha_2 ... \alpha_k ...}^{C} , \alpha_i \in \{0,1,...,i\} \quad i \in A, \quad \alpha_i =0 \quad i \in B \}
$$
Let us show that $\dim_{H} (V) \ne \dim_{H}(F_{\xi} (V) )$ .

2. Let's calculate $\dim_{H} (V)$ .

Consider the random variable $\psi = \Delta_{\psi_1 \psi_2 ... \psi_k ...}^{C}$ , where
\\

\begin{tabular}{|c|c|c|c|c|}
\hline
{$\xi_k$} & 0 & 1 & ... & $k$ \\
\hline
 & $\frac {1}{k+1}$ & $\frac {1}{k+1}$ & ...& $\frac {1}{k+1}$  \\
\hline
\end{tabular}
 \quad $ \forall k \in A :=\{n:n \neq 10^s , s \in N \}$
\\
\\

\begin{tabular}{|c|c|c|c|c|}
\hline
{$\xi_k$} & 0 & 1 & ... & $k$ \\
\hline
   & $1$ & $0$ & ...& $0$  \\
\hline
\end{tabular}
\quad $\forall k\in \bar{A}=B$
\\
\\

It is easy to see that the set $V$ is the spectrum of the random variable $\psi$.

If
$$
\sum_{k=1}^{\infty} \left( \frac {\ln{n_k}}{ \ln (n_{1}\cdot n_{2}\cdot...\cdot n_{k-1})} \right)^2  = \sum_{k=1}^{\infty} \left( \frac{\ln{(k+1)}} {\ln {(k!)}} \right)^2 <+\infty,
$$
then the Hausdorff-Besicovitch dimension of the spectrum can be calculated by the following formula \cite{AILT}:
$$
\dim_{H} S_\xi = \varliminf_{k\to\infty} \frac{\ln {(m_1\cdot m_2 \cdot...\cdot m_k )} }{\ln {(n_1\cdot n_2 \cdot...\cdot n_k )} },
$$
where $m_k$ is the number of non-zero probabilities among $p_{0k}, p_{1k} , ... , p_{(n_k-1)k}$ .
$$
\dim_{H} S_\psi = \varliminf_{k\to\infty} \frac{\ln{(m_1\cdot m_2 \cdot...\cdot m_k )}}{\ln{(n_1\cdot n_2 \cdot...\cdot n_k ) }}=
$$

$$
=\varliminf_{k\to\infty} \frac { \ln ( 2 \cdot 3 \cdot ... \cdot 9 \cdot 10 \cdot 1 \cdot 12 \cdot ...\cdot 100 \cdot 1 \cdot 102 \cdot ... \cdot 10^{k} \cdot 1 ) } { \ln ( n_1 \cdot n_2 \cdot ... \cdot n_{10^k} ) }=
$$

$$
=\varliminf_{k\to\infty} \frac {\ln{(2\cdot3\cdot...\cdot10^{k}\cdot(10^{k}+1))}-\ln{(11\cdot\ 101\cdot...\cdot(10^k +1))}} {\ln{(n_1\cdot n_2 \cdot...\cdot  n_{{10}^k} )} }=
$$
$$
=\varliminf_{k\to\infty} \left(1- \frac{\ln({11\cdot\ 101\cdot...\cdot(10^k +1)})}{\ln{(n_1\cdot n_2 \cdot...\cdot  n_{{10}^k} ) }}\right)
$$
Since
$$
\frac{\ln({11\cdot\ 101\cdot 1001\cdot...\cdot(10^k +1)})}{\ln{(n_1\cdot n_2 \cdot...\cdot  n_{{10}^k} ) }} \le \frac{\ln({10\cdot\ 100\cdot 1000\cdot...\cdot10^{k +1})}}{\ln{(2\cdot 3 \cdot 4 \cdot...\cdot (10^k +1 ) ) }} =
$$
$$
=\frac{\ln({10\cdot\ 10^2\cdot 10^3\cdot...\cdot10^{k +1})}}{\ln{( (10^k +1 )! ) }} =
$$

$$
=\frac{\frac{k(k+1)}{2} \ln 10}{\ln ( \sqrt{2\pi(10^k +1)})+(10^k +1) \ln ({\frac{10^k +1}{e}})+\theta_k} \to 0 (k \to  \infty ),
$$

because $0<\theta_k<\frac{1}{12k}$.

Therefore, 
$$
=\varliminf_{k\to\infty} \left({1- \frac{\ln({11\cdot\ 101\cdot...\cdot(10^k +1)})}{\ln{(n_1\cdot n_2 \cdot...\cdot n_{{10}^k} )}}}\right) =1.
$$
$$
\dim_{H} ( V ) =\dim_{H} S_\psi =1
$$
3. Now let`s calculate $\dim_{H} ( F_\xi ( V) )$.

For any $x \in V$ consider the limit
$$
\lim_{k\to\infty} \frac{\ln{\lambda(\Delta_{\alpha_1 (x) \alpha_2 (x) ...\alpha_k (x)} )}}{\ln{\mu_\xi (\Delta_{\alpha_1 (x) \alpha_2 (x) ...\alpha_k (x)} )}} = \lim_{k\to\infty} \frac{-\ln{(n_1\cdot n_2 \cdot...\cdot n_{k} )}}{\ln( p_{\alpha_1 (x)1} \cdot p_{\alpha_2 (x)2}\cdot ...\cdot p_{\alpha_k (x)k})}
$$

Let us show that this limit exists for any $x\in V$ .

Let
$$
b_k (x):= \frac{\ln{\lambda(\Delta_{\alpha_1 (x) \alpha_2 (x) ...\alpha_k (x)} )}}{\ln{\mu_\xi (\Delta_{\alpha_1 (x) \alpha_2 (x) ...\alpha_k (x)} )}}=\frac{-\ln{(n_1\cdot n_2 \cdot...\cdot n_k ) }} {\ln{ (p_{\alpha_1 (x)1} \cdot p_{\alpha_2 (x)2}\cdot ...\cdot p_{\alpha_k (x)k})}}.
$$

Let`s consider $\{b_k (x)\}$:
$$
b_1 (x) = \frac {-\ln(n_1)}{\ln(p_{\alpha_{1} (x)1})} =1 , \quad  \forall x \in V
$$
$$
b_2 (x) = \frac {-\ln(n_1\cdot n_2)}{\ln(p_{\alpha_1 (x)1}\cdot p_{\alpha_2 (x)2})} =1 , \quad  \forall x \in V
$$
$$
\vdots
$$
$$
b_9 (x) = \frac {-\ln(n_1\cdot...\cdot n_9)}{\ln(p_{\alpha_1 (x)1}\cdot...\cdot p_{\alpha_9 (x)9})} =1 , \quad  \forall x \in V
$$
$$
b_{10} (x) = \frac {-\ln(n_1\cdot...\cdot n_9\cdot n_{10} )}{\ln(p_{\alpha_1 (x)1}\cdot...\cdot p_{\alpha_9 (x)9}\cdot p_ {\alpha_{10} (x)10})}=\frac{\ln(2\cdot...\cdot10\cdot11)}{\ln(2\cdot...\cdot10\cdot{{10}^{{{10}^{10}}}} )}<1 , \quad  \forall x \in V
$$
$$
b_{11} (x) = \frac {-\ln(n_1\cdot...\cdot n_{10}\cdot n_{11} )}{\ln(p_{\alpha_1 (x)1}\cdot...\cdot p_{\alpha_{10} (x)10}\cdot p_ {\alpha_{11} (x)11})}=\frac{\ln(2\cdot...\cdot11\cdot12)}{\ln(2\cdot...\cdot{{10}^{{{10}^{10}}}}\cdot12 )} <1 , \quad  \forall x \in V
$$

Hence

$$
b_1(x)=b_2(x)=...=b_9(x)>b_{10}(x)
$$

$$
b_{10}(x)<b_{11}(x)<...<b_{99}(x)>b_{100}(x)
$$

$$
b_{100}(x)<b_{101}(x)<...<b_{999}(x)>b_{1000}(x)
$$

$$
\vdots
$$

$$
b_{1000}(x)<b_{1001}(x)<...<b_{9999}(x)>b_{10000}(x)
$$

$$
\vdots
$$

$$
b_{10^k}(x)<b_{10^k +1}(x)<...<b_{10^{k+1} -1}(x)>b_{10^{k+1}}(x), \quad  \forall k \in N
$$

To show that $\lim\limits_{k\to\infty} b_k (x)$ exists and is equal to 0, we calculate:

$$
1) \varliminf_{k\to\infty} b_k (x)= \varliminf_{k\to\infty} b_{{10}^k} (x)= \varliminf_{k\to\infty} \frac{-\ln{(n_1\cdot n_2 \cdot...\cdot n_{{10}^k} ) }}{\ln{ \left(p_{\alpha_1 (x)1} \cdot p_{\alpha_2 (x)2}\cdot ...\cdot p_{\alpha_{{10}^k} (x){10}^k}\right)}}
$$
Since
$$
\ln \left( \frac{1}{p_{\alpha_1 (x)1}} \cdot \frac{1}{p_{\alpha_2 (x)2}}\cdot ...\cdot \frac{1}{p_{\alpha_{{10}^k} (x){10}^k}} \right)=
$$
$$
=\ln { 2\cdot3\cdot...\cdot10\cdot{10}^{{10}^{{{10}^{{10}^1}}}}}\cdot 12\cdot...\cdot 99 \cdot 100\cdot{10}^{{10}^{{{10}^{{10}^2}}}}\cdot...\cdot {10}^{{10}^{{{10}^{{10}^k}}}}>
$$
$$
> \ln{ {10}^{{10}^{{{10}^{{10}^k}}}} }={10}^{{{10}^{{10}^k}}} \ln {10}
$$
and
$$
\ln{(n_1\cdot n_2 \cdot...\cdot n_{{10}^k} )}=\ln{({10}^k+1)!}<\ln{({10}^k)!}=
$$

$$
=\ln{\left(\sqrt{2\pi{10}^k}\cdot {\left(\frac{{10}^k}{e}\right)}^{{10}^k} \cdot \theta_k\right)}=\ln\left(\sqrt{2\pi{10}^k}\right) + {10}^k \ln{ \left(\frac{{10}^k}{e}\right)} +\ln \theta_k , 
$$

where $0<\theta_k<\frac{1}{12k}$,

we have
$$
\frac{\ln{(n_1\cdot n_2 \cdot...\cdot n_{{10}^k} ) }}{\ln{ 
\left( \frac{1}{p_{\alpha_1 (x)1}} \cdot \frac{1}{p_{\alpha_2 (x)2}}\cdot ...\cdot \frac{1}{p_{\alpha_{{10}^k} (x){10}^k}}\right) }} < \frac{\ln\left(\sqrt{2\pi{10}^k}\right) + {10}^k \ln{ \left(\frac{{10}^k}{e}\right)} +\ln \theta_k} {{10}^{{{10}^{{10}^k}}} \ln {10}} \to 0 \ (k \to\infty)
$$
Therefore
$$
\varliminf_{k\to\infty} \frac{\ln{(n_1\cdot n_2 \cdot...\cdot n_{10^k} ) }}{\ln \left( \frac{1}{p_{\alpha_1 (x)1}} \cdot \frac{1}{p_{\alpha_2 (x)2}}\cdot ...\cdot \frac{1}{p_{\alpha_{10^k} (x) 10^k}}\right)} =0.
$$

$$
2) \varlimsup_{k\to\infty} b_k (x)=\varlimsup_{k\to\infty} b_{({{10}^k -1})} (x)= \varlimsup_{k\to\infty}\frac{-\ln{(n_1\cdot n_2 \cdot...\cdot n_{(10^k -1)} ) }}{\ln \left(p_{\alpha_1 (x)1} \cdot p_{\alpha_2 (x)2}\cdot ...\cdot p_{\alpha_{(10^k -1)} (x)(10^k -1)}\right)} 
$$
Since

$$
\ln \left( \frac{1}{p_{\alpha_1 (x)1}} \cdot \frac{1}{p_{\alpha_2 (x)2}}\cdot ...\cdot \frac{1}{p_{\alpha_{(10^k -1)} (x)(10^k -1)}} \right)=
$$

$$
=\ln ({2\cdot...\cdot10\cdot{10}^{{10}^{{{10}^{{10}^1}}}}}\cdot 12\cdot... \cdot {10}^2\cdot{10}^{{10}^{{{10}^{{10}^2}}}}\cdot...\cdot{10}^{k-1}\cdot {10}^{{10}^{{{10}^{{10}^k-1}}}}\cdot...\cdot{10}^k)>
$$
$$
> \ln{ {10}^{{10}^{{{10}^{{10}^k-1}}}} }={10}^{{{10}^{{10}^k-1}}} \ln {10}
$$

and
$$
\ln \left(n_1\cdot n_2 \cdot...\cdot n_{({10}^k -1)} \right) =\ln {({10}^k)}!=
$$

$$
=\ln \left(\sqrt{2\pi{10}^k}\cdot {\left(\frac{{10}^k}{e}\right)}^{{10}^k} \cdot \theta_k \right)=\ln\left(\sqrt{2\pi{10}^k}\right) + {10}^k \ln{ \left(\frac{{10}^k}{e}\right)} +\ln \theta_k  , 
$$
where $0<\theta_k<\frac{1}{12k}$,

we have
$$
\varlimsup_{k\to\infty}\frac{-\ln{\left(n_1\cdot n_2 \cdot...\cdot n_{(10^k -1)} \right) }}{\ln \left(p_{\alpha_1 (x)1} \cdot p_{\alpha_2 (x)2}\cdot ...\cdot p_{\alpha_{(10^k -1)} (x)(10^k -1)}\right)} \le \varlimsup_{k\to\infty}\frac{\ln\left(\sqrt{2\pi{10}^k}\right) + {10}^k \ln{ \left(\frac{{10}^k}{e}\right)} +\ln \theta_k}{{10}^{{{10}^{{10}^k-1}}} \ln {10}} =0
$$

If
$$
0=\varliminf_{k\to\infty} b_k (x)\le \varlimsup_{k\to\infty} b_k (x) \le 0 ,
$$
then
$$
\lim_{k\to\infty} b_k =0 .
$$
According to Billingsley's theorem\cite{Billingsley 1961}: 
$$
If  \quad V=\left\{x: \lim_{k\to\infty} \frac{\ln{\lambda(\Delta_{\alpha_1 (x) \alpha_2 (x) ...\alpha_k (x)} )}}{\ln{\mu_\xi (\Delta_{\alpha_1 (x) \alpha_2 (x) ...\alpha_k (x)}  )}} =\delta \right\} ,  then  \quad \dim{H}(V,\mu_\xi,\Phi)=\delta\dim_{H}(V,\lambda,\Phi).
$$ 

1. \vspace{1cm} $ \dim_{H}(V,\lambda, \Phi )=\inf\{\alpha:H^\alpha (V,\lambda,\Phi(C))=0\} $ ,

where $\Phi(C)$  -- family of cylinders of the Cantor expansion.
$$
H_{\varepsilon}^{\alpha} ((V,\lambda,\Phi(C))=\inf_{|E_j|\le \varepsilon} \sum_{j} {\lambda(E_j)}^\alpha , \quad E_j \in \Phi(C))
$$
It is not difficult to prove that, if the sequence $\{n_k\}$  forms an arithmetic progression, then $\Phi(C)$ is a faithful family for computing the Hausdorff-Besicovitch dimension, i.e.
$$
\dim_{H}(E)=\dim_{H} (E, \Phi(C)), \quad \forall E \subset [0;1] .
$$
Therefore
$$
\dim_{H}(V,\lambda,\Phi(C))=\dim_{H}(V,\lambda)=\dim_{H}(V)
$$

2. \vspace{1cm} $ \dim{H}(V,\mu_\xi, \Phi )=\inf\{\alpha:H^\alpha (V,\mu_\xi,\Phi(C))=0\}$

$$
H_{\varepsilon}^{\alpha} (V,\mu_\xi,\Phi(C))=\inf_{|V_j|\le \varepsilon} \sum_{j} {\,\mu_\xi(V_j)}^\alpha , \quad V_j \in \Phi(C))
$$

It is easy to see that

$$
\mu_\xi (\Delta_{\alpha_1 (x) \alpha_2 (x) ... \alpha_k (x)} ) = p_{\alpha_1 (x)1} \cdot p_{\alpha_2 (x)2} \cdot ... \cdot p_{ \alpha_k (x)k} = $$ 

$$ = \left| \Delta_{\alpha_1 (x) \alpha_2 (x) ... \alpha_k (x)}^{\tilde{P}} \right| =
 \left| { F_\xi \left( \Delta_{\alpha_1 (x) \alpha_2 (x) ... \alpha_k (x)}^C \right) } \right| .
$$

Since $V$ was covered by $\{V_j \}$, we conclude that $F_\xi (V)$ can be covered by $\{F_\xi (V_j)\}$. 

Therefore
$$
\dim_{H}(V,\mu_\xi,\Phi(C))=\dim_{H}(F_\xi (V),\lambda,F_\xi(\Phi(C)))= \dim_{H} (F_\xi(V),\Phi'),
$$
where  $\Phi'=F_\xi (\Phi(C))\ $ .

According to the article by V. Vasylenko, V. Misky, G. Torbin \cite{VMT} at $n_k\le n_0$ and $p_{ik}\ge p_0 > 0 : \Phi'=\Phi' (\tilde{Q})$ -- is faithful.

Hence, $\dim_{H} (V,\mu_\xi )=\dim_{H} (F_\xi (V) )$ .

So, 
$$
\dim_{H} (F_\xi (V) )=0\cdot \dim_{H} (V)
$$
$$
\dim_{H} (F_\xi (V) )=0 
$$.

That is $\dim_{H} (V)=1\ne 0=\dim_{H} (V)$.

This means that $F_\xi$ is not a DP-transformation.
\\
\textbf{Acknowledgment} This work was partly supported by a grant from the Simons Foundation (1290607, Torbin G.) and by the Ministry of Science and Education of Ukraine (projects "Multifractal analysis of probability distributions and its application for modeling of complex dynamical systems" no.0122U000048  and "Dynamics of complex systems on different time scales"  no.0120U101662 .

\end{document}